\documentclass[a4paper,12pt,twoside]{amsart}
\usepackage{amssymb}

\usepackage{amsfonts}
\usepackage{amsfonts,amssymb,amscd}
\usepackage{graphicx}
 \usepackage[left=2.3cm,top=3.3cm,right=2.3cm,bottom=2.7cm]{geometry}
\usepackage{mathrsfs}
\usepackage{xfrac}
\let\mathcal\mathscr
\usepackage[colorlinks=true,citecolor=blue, urlcolor=blue, breaklinks, linkcolor=black]{hyperref}

\usepackage[all,ps,cmtip]{xy}

\DeclareRobustCommand{\SkipTocEntry}[5]{}

\def\llra{\hbox to 10mm{\rightarrowfill}}

\def\lllra{\hbox to 15mm{\rightarrowfill}}

\def\PB{{\widehat B}}

\def\phi{{\varphi}}

\def\cI{\mathcal{I}}

\def\cF{\mathcal{F}}
\def\cL{\mathcal{L}}
\def\cO{\mathcal{O}}

\def\cP{\mathcal{P}}

\def\cQ{\mathcal{Q}}

\let\tilde\widetilde

\DeclareMathOperator{\Pic}{Pic}

\DeclareMathOperator{\vol}{vol}

\newtheorem{lemm}{Lemma}[section]
\newtheorem{theo}[lemm]{Theorem}
\newtheorem{coro}[lemm]{Corollary}
\newtheorem{prop}[lemm]{Proposition}

\theoremstyle{definition}
\newtheorem{defi}[lemm]{Definition}
\newtheorem{rema}[lemm]{Remark}

\newtheorem{exam}[lemm]{Example}

\newtheorem{qu}[lemm]{Question}

\theoremstyle{remark}
\newtheorem*{remark*}{Remark}
\newtheorem*{note*}{Note}

\begin{document}
\title{On Severi type inequalities}
\author{Zhi Jiang}
\address{Shanghai center for mathematical sciences, Xingjiangwan campus, Fudan University, Shanghai 200438, P. R. China}
\email{zhijiang@fudan.edu.cn}
 \thanks{The author is partially supported by the program “Recruitment of global experts”, NSFC grants No. 11871155 and No. 11731004.}
\maketitle

\setlength{\parskip}{.1 in}

\begin{abstract}
We study Severi type inequalities for big line bundles on irregular varieties via cohomological rank functions. We show that these Severi type inequalities are related to some natural defined birational invariants of the general fibers of the Albanese morphisms.
As an application, we show that the volume of an irregular threefold of general type is at least $\frac{3}{8}$.
We also show that the volume of a smooth projective variety $X$ of general type and of maximal Albanese dimension is at least $2(\dim X)!$. Moreover, if $\vol(X)=2(\dim X)!$,  the canonical model  of $X$ is a flat double cover of a principally polarized abelian variety $(A, \Theta)$ branched over some divisor $D\in |2\Theta|$.
\end{abstract}

\section{Introduction}

Xiao proved an inequality among certain Chern numbers for surfaces with fibrations to curves (\cite{X}), which is now called Xiao's slope inequality. Assume that $f: X\rightarrow C$ is a relatively minimal and not locally trivial fibration from a smooth projective surface to a smooth curve.  Let $g$ be the genus of a general fiber of $f$ and let $b$ be the genus of $C$. Then Xiao showed that
\begin{eqnarray*}
K_{S/C}^2/\big(\chi(\cO_S)-(g-1)(b-1)\big)\geq 4(1-\frac{1}{g}).
\end{eqnarray*}
Note that $\chi(\cO_S)=\chi(\omega_S)$ is a birational invariant of $S$. If $S$ is minimal surface, then $K_{S/C}^2=K_S^2-2K_S\cdot f^*K_C=\vol(S)-8(g-1)(b-1)$ and we can rewrite Xiao's inequality as an inequality between birational invariants:
$$\vol(S)\geq \frac{4(g-1)}{g}\chi(\omega_S)+\frac{(g^2-1)(b-1)}{g}.$$

 This inequality played an essential role in Pardini's proof of Severi's inequality of surfaces (\cite{par}). Assume that $S$ is a surface of general type, of maximal Albanese dimension, then $\vol(S)\geq 4\chi(\omega_X)$. Indeed, Pardini noticed that one can construct \'etale covers $S_M\rightarrow S$ such that $\vol(S_M)=M^{2q(S)}\vol(S)$ and $\chi(\omega_{S_M})=M^{2q(S)}\chi(\omega_S)$ and fibrations $f_M: S_M\rightarrow \mathbb P^1$ such that the genus of a general fiber of $f_M$ is $O(M^{2q(S)-2})$, for each $M\in\mathbb N$. Applying Xiao's equality for these fibrations $f_M$, Pardini proved that $\vol(S)\geq 4\chi(\omega_S)$, which is called Severi's inequality.

 Severi's inequality was extended to varieties of maximal Albanese dimensions of higher dimensions  independently by Barja \cite{bar} and Zhang \cite{zh1}.
 Let $X$ be a smooth projective variety of general type and of maximal Albanese dimension, then $\vol(X)\geq 2(\dim X)!\chi(\omega_X)$.
 A crucial feature in both proofs is that, by generic vanishing, one could regard $\chi(\omega_X)$ as $h^0(X, \omega_X\otimes Q)$, where $Q\in \Pic^0(X)$ is a general numerically trivial line bundle, and hence $\chi(\omega_X)\geq 0$. Note that this inequality gives a natural lower bound for $\vol(X)$ when $\chi(\omega_X)>0$,
 however, $\chi(\omega_X)$ could be $0$ in dimension $\geq 3$ (\cite{EL}).

  A refined Severi's inequality of surfaces was obtained by Lu and Zuo recently (\cite{LZ}). They proved that $\vol(S)\geq \min\{\frac{9}{2}\chi(\omega_X), 4\chi(\omega_X)+4(q(X)-2))\}$, which is crucial in their classification of surfaces on the Severi line. In \cite{BPS}, Barja, Pardini, and Stoppino introduced  continuous rank functions on abelian varieties and proved several important results, including a simple proof of Barja and Zhang's higher dimensional Severi's inequality and various refinements of Severi's inequality in all dimensions.

The two questions which we are interested in in this article are the followings.
\begin{itemize}
\item[(1)]  Does there exist Severi type inequalities on general irregular varieties ?
\item[(2)]  Does there exist refined Severi type inequalities for varieties of maximal Albanese dimensions, taking into account the irregularities of the varieties ?
\end{itemize}

Note that Barja and Zhang have independently proved Severi type inequalities on varieties of Albanese fiber dimension one.

\begin{theo}[Barja \cite{bar}]
Let $f: X\rightarrow A$ be a morphism from a smooth projective variety of general type to an abelian variety. Assume that $\dim X-\dim f(X)=1$. Let $C$ be a connected component of a general fiber of $f$. Then $$\vol(X)\geq (\dim X-1)! \chi(f_*\omega_X).$$
\end{theo}

This result was improved by Zhang.
\begin{theo}[Zhang \cite{zh1}]
Under the same assumption, let $g$ be the genus of $C$, then
 $$\vol(X)\geq 2(\dim X)! \frac{g-1}{g+\dim X-2}\chi(f_*\omega_X).$$
\end{theo}

In this article, we follow the method of Barja, Pardini, and Stoppino \cite{BPS}, who applied continuous rank functions to study Severi type inequalities on varieties of maximal Albanese dimensions.
In \cite{JP}, cohomological rank functions on abelian varieties, which are generalization of continuous rank functions, were defined and studied. Let $(A, \underline{l})$ be a polarized abelian variety of dimension $n$ and let $\cF\in D^b(A)$ be an object in the derived category of bounded complexes of coherent sheaves,   the $i$-th cohomological rank function is defined to be a continuous function $$h^i_{\cF}(x\underline{l}): \mathbb R\rightarrow \mathbb R$$ such that $$h^i_{\cF}(x\underline{l})=\frac{1}{M^{2n}}h^i(A, \mu_M^*\cF\otimes L^{M^2x}\otimes Q),$$ where $x\in \mathbb Q$, $M^2x\in \mathbb Z$, and $Q\in \Pic^0(A)$ is a very general numerically trivial line bundle. In \cite{JP}, some basic properties of cohomological rank functions have been studied. We know that they are always continuous functions and we can compute the left and right derivatives at each rational point.
 Hence it is possible to apply these functions to study Severi type inequalities for general irregular varieties.

Let $f: X\rightarrow A$ be a morphism from a smooth projective variety to an abelian variety and let $F$ be a connected component of a general fiber of $f$ over its image. Our main result is a Severi type inequality of $X$ which depends on Clifford or Noether type inequalities on $F$.
 Here is one simple version and see Theorem \ref{canonical-bundle-general} for the full statement.

 \begin{theo}Assume that the linear system $|K_F|$ induces a generically finite map of $F$. Then $$\vol(F)\geq 2(\dim X-\dim F)!\chi(f_*\omega_X).$$
 \end{theo}

Severi type inequalities naturally provide estimates of irregular varieties of general type. We know by the seminar work of Hacon-McKernan \cite{HM}, Takayama \cite{T}, Tsuji \cite{Ts} that in each dimension $n$, there exists a positive lower bound $\epsilon(n)$ of the set of the volumes of smooth projective varieties of general type of dimension $n$. In general this lower bound is quite difficult to compute due to the singularities of minimal models of varieties of general type in dimension $\geq 3$. J.A. Chen and M. Chen proved that $\epsilon(3)\geq \frac{1}{1680}$ in \cite{CC2}. We can also define $\epsilon_I(n)$ the lower bound of the set of the volumes of irregular smooth projective varieties of general type of dimension $n$. J.A. Chen and M. Chen showed that $\epsilon_I(n)\geq \frac{1}{22}$ (see \cite{CC}).

Define $\epsilon_{\mathrm{mAd}}(n)$ to be the lower bound of the set of the volumes of  smooth projective varieties of maximal Albanese dimension, of general type, and of dimension $n$. Then we have the obvious bound $$\epsilon_I(n)\leq \mathrm{min}\{\binom{n}{k}\epsilon_{\mathrm{mAd}}(k)\epsilon(n-k)\mid 1\leq k\leq n\}.$$ It would be interesting to prove a similar lower bound for $\epsilon_I(n)$.

Our first application of Severi type inequality is to show that $\epsilon_{\mathrm{mAd}}(n)=2n!$. In \cite{BPS2}, the authors studied varieties $X$ on the Severi line, namely $\vol(X)=2(\dim X)!$ and showed that $a_X$ is always a degree $2$ cover. With the help of this  result, we can completely describe $n$-dimensional varieties $X$ of maximal Albanese dimension with $\vol(X)=2n!$.
\begin{theo}
Assume that $X$ is a smooth projective variety of maximal Albanese dimension and of general type. Then $\vol(X)\geq 2(\dim X)!$. If the equality holds, then $A_X$ admits a principal polarization $\Theta$ and the canonical model $X_{\mathrm{can}}$ is a flat double cover of $A_X$ branched over a divisor $D\in |2\Theta|$.
\end{theo}

We can also provide a better bound for $\epsilon_I(3)$, which is indeed close to the optimal bound.
 \begin{theo}$\epsilon_I(3)\geq \frac{3}{8}$.
\end{theo}

\subsection*{Acknowledgements}
We thank Fabrizio Catanese, Yong Hu, Mart\'i Lahoz, Giuseppe Pareschi and Tong Zhang for stimulating conversations. We thank Olivier Debarre for reading the draft carefully. Parts of this work were written during the author's visits to the Graduate School of Mathematical Sciences in the University of Tokyo and National Center for Theoretical Sciences in Taipei and we thank  Jheng-Jie Chen, Jungkai Alfred Chen, Yoshinori Gongyo, and Yusuke Nakamura for the warm hospitality.
\section{Preliminaries}
\subsection{Maximal continuously globally generated subsheaves}
\begin{defi} A coherent sheaf $\cF$ on an abelian variety $A$ is said to be continuously globally generated if for any open subset $U$ of $\Pic^0(A)$, the evaluation map $$\mathrm{ev}_U: \bigoplus_{Q\in U}H^0(A, \cF\otimes Q)\otimes Q^{-1}\rightarrow \cF$$ is surjective.
\end{defi}
\begin{rema} Continuously globally generated sheaves were introduced by Pareschi and Popa in \cite{PP}. They proved that M-regular sheaves are indeed continuously globally generated. Debarre \cite{D} gave a nice characterization: a coherent sheaf $\cF$ is  continuously globally generated if and only if there exists $M\in \mathbb Z$ such that $\mu_M^*\cF\otimes Q$ is  globally generated for any $Q\in \Pic^0(A)$. Moreover, a continuously globally generated sheaf is ample.
\end{rema}
\begin{lemm}\label{etale}
Let $\cF$ be a coherent sheaf on $A$. Then there exists a continuously globally generated  subsheaf $\cF_c$ of $\cF$ such that $h^0(A, \cF_c\otimes Q)=h^0(A, \cF\otimes Q)$ for $Q\in \Pic^0(A)$ general. Moreover, for any isogeny $\mu: A'\rightarrow A$ between abelian varieties, $(\mu^*\cF)_c=\mu^*\cF_c$.
\end{lemm}
\begin{proof}
Let $U_0$ be the open subset of $\Pic^0(A)$ consisting of $Q$ such that $h^0(A, \cF\otimes Q)$  takes the minimum value. Let $\cF_c $ be the image of the evaluation map $$ev_{U_0}: \bigoplus_{Q\in U_0}H^0(A, \cF\otimes Q)\otimes Q^{-1}\rightarrow \cF.$$ Then it is clear that $h^0(A, \cF_c\otimes Q)=h^0(A, \cF\otimes Q)$ for $Q\in \Pic^0(A)$ general.  It suffices to prove that $\cF_c$ is continuously globally generated. We just need to show that for any $V\subset U_0$, the image $\cF_V$ of $\mathrm{ev}_V$ is exactly $\cF_{c}$. Note that $\cF_V\subset \cF_c$. Moreover, for $Q\in V$, $h^0(A, \cF_V\otimes Q)=h^0(A, \cF_c\otimes Q)$ is the generic value. Hence by the semicontinuity theorem, for all $Q\in U_0$, we still have $h^0(A, \cF_V\otimes Q)=h^0(A, \cF_c\otimes Q)$. Thus the image of $\mathrm{ev}_V$ is the image of $\mathrm{ev}_{U_0}$.

For the second statement, we note that $\mu^*\cF_c\subset (\mu^*\cF)_c$ and $$h^0(A', \mu^*\cF_c\otimes Q')=(\deg\mu) h^0(A, \cF_c\otimes Q)=(\deg\mu) h^0(A, \cF\otimes Q)=h^0(A', \mu^*\cF\otimes Q'),$$ where $Q\in \Pic^0(A)$ general and $\mu^*Q=Q'$.  Hence  $(\mu^*\cF)_c=\mu^*\cF_c$.
\end{proof}

\begin{defi}
We call $\cF_c$ the maximal continuously globally generated subsheaf of $\cF$.
\end{defi}
\begin{rema}\label{decom} By the decomposition theorems in  \cite{CJ} and \cite{PPS}, we know that given a morphism $f: X\rightarrow A$ from a smooth projective morphism to an abelian variety, there exists a canonical decomposition
\begin{eqnarray}\label{decomposition}
f_*\omega_X=\cF_X\bigoplus_{p_B}\bigoplus_ip_B^*\cF_{B,i}\otimes Q_{B,i},
\end{eqnarray}
where $\cF_X$ is either $0$ or a M-regular sheaf on $A$, $p_B: A\rightarrow B$ are quotients between abelian varieties, $\cF_{B,i}$ are M-regular on $B,$ and $Q_{B,i} $ are torsion line bundles on $A$.
It is clear that $(f_*\omega_X)_c=\cF_X$.
\end{rema}

\begin{rema}
It is worth noting that M-regular sheaves are continuously globally generated but the converse is not true in general. The following is an example.

Let $i: C\hookrightarrow A_C$ be the Abel-Jacobi embedding of a genus $g\geq 2$ curve. Let $\Theta$ be the theta divisor of $A_C$. Then we know by \cite{JP} that $h^0_{i_*\cO_C}(x\Theta)=x^g$ and $h^1_{i_*\cO_C}(x\Theta)=g-1-gx+x^g$ for $0\leq x\leq 1$. This statement implies that for any $k\geq 2$ integer, let $i_k: C_k\hookrightarrow A_C$ be the pullback of $i: C\hookrightarrow A_C$ be the isogeny $\mu_k$. Then $h^0(A_C, i_{k*}\cO_{C_k}\otimes \Theta^k\otimes Q)=k^g$ and $h^1(A_C, i_{k*}\cO_{C_k}\otimes \Theta^k\otimes Q)=k^{2g}(g-1)-gk^{2g-1}+k^g$.

Let $\cL=i_{k*}\cO_{C_k}\otimes \Theta^k$ and let $\cL_c$ be its maximal continuously globally generated subsheaf. Then $h^0(A, \cL_c\otimes Q)=h^0(A, \cL\otimes Q)>0$ for $Q\in\Pic^0(A)$ general. Note that the cokernel of $\cL_c\hookrightarrow \cL$  is a $0$-dimensional sheaf. Hence $h^1(A, \cL_c\otimes Q)>0$ for $Q\in\Pic^0(A)$ general. Thus $\cL_c$ is continuously globally generated but is not M-regular.
\end{rema}

More generally, for a coherent sheaf $\cF$ on $A$ and a line bundle $H$ of $A$, we consider the $\mathbb Q$-twisted sheaf $\cF\otimes H^x$ for $x\in \mathbb Q$ as in \cite{JP}. For $M\in \mathbb Z$ such that $M^2x\in \mathbb Z$, the sheaf $(\mu_M^*\cF\otimes H^{M^2x})_c$ is well-defined. Moreover, by Lemma \ref{etale}, $(\mu_{NM}^*\cF\otimes H^{N^2M^2x})_c=\mu_N^*(\mu_M^*\cF\otimes H^{M^2x})_c$ for any $N\in\mathbb Z$. Hence we will formally define $(\cF\otimes H^x)_c$ upto abelian \'etale covers.

We will apply the following notations.

Given $f: X\rightarrow A$  a primitive morphism from a smooth projective variety to an abelian variety. For any coherent sheaf $\cQ$ on $X$, we denote by $\cQ_c$ the image of the natural map $f^*(f_*\cQ)_c\rightarrow \cQ$.

 If $\cQ$ is a line bundle on $X$, we always take a birational modification such that $\cQ_c$ is also locally free and hence is a nef line bundle. Indeed, $\mu_M^*\cQ$ is globally generated for some $M\in \mathbb Z$. Similarly, we can define $(\cQ\otimes f^*H^x)_c$ for $H$ a line bundle on $A$ and $x\in\mathbb Q$, upto abelian \'etale covers and birational modifications.

\subsection{The eventual maps}
 \begin{defi}
 Given $f: X\rightarrow A$ a primitive morphism from a smooth projective variety to an abelian variety. Let $\cL$ be a line bundle on $X$ with $\cL_c$ non-zero. We denote by $\varphi_{\cL}: X \dashrightarrow\mathbb Z_{\cL}\hookrightarrow \mathbb{P}_A((f_*\cL)_c)$ the relative evaluation map, where $Z_{\cL}$ is the image of $\varphi_{\cL}$. We call $\varphi_{\cL}$ the eventual map of $\cL$ (with respect to $f$).
 \end{defi}

By the following lemma,  our definition of the eventual map is compatible with the one of Barja, Pardini, and Stoppino in \cite{BPS}. Following \cite{BPS}, we will call $\varphi_{K_X}$ the eventual paracanonical map of $X$.

\begin{lemm}\label{eventual}Let $M\in \mathbb N_{>0}$. Consider the \'etale base change
\begin{eqnarray*}
\xymatrix{
X_M\ar[r]^{\mu_M}\ar[d]^{f_M} & X\ar[d]^f\\
A\ar[r]^{\pi_M} & A.}
\end{eqnarray*}
Let $\varphi_M: X_M\dashrightarrow Z_M\hookrightarrow \mathbb P(H^0(X_M, \mu_M^*\cL\otimes f_M^*Q))$ be the natural evaluation map for $Q\in\Pic^0(A)$ general. Assume that $M$ is sufficiently large, $\varphi_M$ is birationally equivalent to the base change by $\mu_M$ of $\varphi_{\cL}$.
\end{lemm}
\begin{proof}
By assumption $\cL_c$ is not zero. Hence $H^0(X, \cL\otimes f^*Q)=H^0(X, \cL_c\otimes f^*Q)\neq 0$ for $Q\in\Pic^0(A)$ general.

Let $\varphi_{\cL, M}: X_M\rightarrow Z_{\cL, M}$ be the base change by $\pi_M$ of $\varphi_{\cL}$.
Then $\varphi_{\cL, M}$ is the relative evaluation map $X_M \dashrightarrow\mathbb Z_{\cL, M}\hookrightarrow \mathbb{P}_A(\pi_{M}^*(f_*\cL)_c)$. By Lemma \ref{etale}, we can regard $\varphi_{\cL, M}$ as $$X_M \dashrightarrow\mathbb Z_{\cL, M}\hookrightarrow \mathbb{P}_A(f_{M*}(\mu_M^*\cL)_c)= \mathbb{P}_A(f_{M*}(\mu_M^*\cL)_c\otimes Q).$$
Since $H^0(X_M, \mu_M^*\cL\otimes Q)=H^0(X_M, \mu_M^*\cL_c\otimes Q)$, we have a natural factorization of $$\varphi_M: X_M\xrightarrow{\varphi_{\cL, M}} Z_{\cL, M} \xrightarrow{\rho} Z_M.$$ We just need to prove that $\rho$ is birational.   Fix a very ample line bundle $H$ on $A$ and let $\underline{h}$ be its class in N\'eron-Severi group,  by \cite{JP} we know from the continuity of cohomological rank functions that $h^0_{f_*\cL_c}(x\underline{h})\neq 0$ for $\epsilon\in \mathbb Q_{>0}$ sufficiently small, which means that for $M$ sufficiently large we have $h^0(A, \pi_M^*f_*\cL_c\otimes H^{-1}\otimes Q)\neq 0$. Hence $f_M$ factors through $\varphi_M$. Moreover, by the  main result \cite{D}, we know that $\pi_M^*f_*\cL_c\otimes Q$ is globally generated for $M$ sufficiently large. Thus $\varphi_M$ and $\varphi_{\cL, M}$ are equivalent restricted on the generic point of $f_M(X_M)$. Hence  $\varphi_M$ and $\varphi_{\cL, M}$ are birationally equivalent.

 \end{proof}
\begin{coro}
$\kappa(\cL)\geq \kappa(\cL_c)=\mathrm{Nm}(\cL_c)=\dim \varphi_{\cL}(X)$.
\end{coro}

 The eventual maps of varieties with maximal Albanese dimension have been studied in \cite{BPS3} and \cite{J}. It turns out that the eventual paracanonical map of a variety of maximal Albanese dimension and of general type is often birational, except when the variety has some special irregular fibration structure. On the other hand, it seems difficult to say something general about the structures of the eventual paracanonical maps for varieties of higer dimensional Albanese fibers.

 \begin{lemm}
 Let $S$ be a surface of general type, of Albanese fiber dimension $1$. Let $C$ be a connected component of a general fiber of the Albanese morphsim. Then the eventual paracanonical map of $S$ is generically finite if $g(C)\geq 5$.
 \end{lemm}
 \begin{proof}
 Let $a_S: S\rightarrow A_S$ be the Albanese morphism of $S$. If $q(S)=\dim A_S\geq 2$, then $a_S(S)$ is a smooth curve of genus equal to $q(S)$ and hence $a_{S*}\omega_S$ is M-regular. The eventual paracanonical map is then generically finite since the canonical map of $C$ is generically finite. We just need to deal with the case that $A_S$ is an elliptic curve. In this case, if $\varphi_{K_S}$  is not generically finite, then  by Remark \ref{decom}, we know that the M-regular part of $a_{S*}\omega_S$ is a line bundle $\cL$ or $0$. On the other hand, $d:=\chi(S, \omega_S)=\chi(A_S, a_{S*}\omega_S)-\chi(A_S, R^1a_{S_*}\omega_S)>0$. Since $A_S$ is an elliptic curve, $R^1a_{S_*}\omega_S=\cO_{A_S}.$ Hence the M-regular part of $a_{S*}\omega_S$ is a line bundle on $A_S$ of degree $d$. We may assume that $S$ is  minimal, and by Xiao's inequality \cite[Lemma 2]{X}, we know that $(2g(C)-2)d< K_{S}^2\leq 9\chi(\omega_S)=9d $. Hence $g(C)\leq 5$.
 \end{proof}
 \begin{qu}
Pignatelli constructed in \cite{P} a minimal surface $S$ with $K_S^2=4$ and $p_g=q=1$ such that $a_{S*}\omega_S$ is the direct sum of an ample line bundle and a torsion line bundle. Hence the genus of a general fiber of the Albanese morphism is $2$. We do not know any examples where the genus of a general fiber of the Albanese morphism equal to $3$ or $4$.
Is it possible to classify those surfaces with $\varphi_{K_S}$ not generically finite ?
  \end{qu}
\subsection{Volume functions and cohomological rank functions}
In this article, we follow the idea in \cite{BPS} to compare the volume functions and the cohomological rank functions which are defined naturally on irregular varieties.

Let $f: X\rightarrow A$ be a primitive morphism from a smooth projective variety to an abelian variety. Fix a very ample polarization $H$ on $A$ and denote by $\underline{h}$ its class.
Let $\cL$ be a big line bundle on $X$. Let $F(t):=\vol_X(\cL+t f^*H)$ and $G(t)=h^0_{f_*\cL}(A, t\underline{h})$.  The volume functions have been studied in \cite{BFJ, ELMNP}.  The definition and  some basic properties of the cohomological rank functions on abelian varieties can be found in \cite{BPS, JP}.

\begin{theo}\label{BFJ}
Let $X$ be a smooth projective variety. Then the volume function $\vol: N^1(X)\rightarrow \mathbb R$ is of class $\mathcal C^1$. Moreover, let $L$ be a big line bundle on $X$, let $H$ be base point free with $D\in |H|$ a general member, then $\frac{d}{dt}\mid_{t=t_0}\vol(L+t_0H)=(\dim X)\vol_{X|D}(L+t_0H)$, where $\vol_{X|D}$ is the restricted volume function.
\end{theo}

The restricted volume function $\vol_{X|D}$ is quite complicated on the pseudo-effective cone of $X$. However, since $D\in |H|$ general, we know that if $M$ is big and nef, then $\vol_{X|D}(M)=\vol_D(M)=(M^{\dim D}\cdot D)_X$  (\cite[Corollary 2.17]{ELMNP}).

It follows from \cite{JP} that $G(t)$ is always continuous  but is not necessarily of class $\mathcal C^1$ in general. The differentiability of $G(t)$ is a bit complicated. At each rational point, the left derivative and right derivative of $g(t)$ exist.
Assume $D$ is an effective divisor of a smooth projective variety $X$. Let $\cF$ be a coherent sheaf on $X$. We denote by $H^0(X|D, \cF)$ the image of the restriction map $H^0(X, \cF)\rightarrow H^0(D, \cF\otimes\cO_D)$ and $h^0(X|D, \cF)=\dim H^0(X|D, \cF)$.
\begin{prop}\label{derivatives} Let  $x_0\in\mathbb Q$. Let $M\in \mathbb Z$ be sufficiently big and divisible such that $M^2x_0\in \mathbb Z$. Let $\pi_M: A\rightarrow A$ be the isogeny induced by the multiplication by $M$ and consider the Cartesian diagram
\begin{eqnarray*}
\xymatrix{
X_M\ar[r]^{\mu_M} \ar[d]^{f_M} & X\ar[d]^f\\
A\ar[r]^{\pi_M} & A.}
\end{eqnarray*}
Since $f$ is primitive, $X_M$ is connected and both $\pi_M$ and $\mu_M$ are \'etale morphisms of degree $M^{2g}$.
Let   $D_M$ be a very general section of $f_M^*H^M$ and let $Q\in\Pic^0(A)$ be very general. Then the right derivative at $x_0$ of the function $G(t)$ is equal to
\[\mathrm{lim}_{M\rightarrow\infty}\frac{1}{M^{2g-1}}h^0(X_M|D_M, \mu_{M}^*(\cL\otimes H^{x_0})\otimes f_M^*(H^M\otimes Q))\]
and the left derivative is equal to
\[\mathrm{lim}_{M\rightarrow\infty}\frac{1}{M^{2g-1}}h^0(X_M|D_M, \mu_{M}^*(\cL\otimes H^{x_0})\otimes f_M^*Q),\]

\end{prop}

Both functions are indeed not easy to compute.
Assume that $t_0=\frac{a}{b}\in \mathbb Q$, we may replace $L+t_0f^*H$ by $(L+t_0f^*H)_c$ to compute its $0$-th cohomological rank function. Note that $(L+t_0f^*H)_c$ is well-defined upto abelian \'etale covers and birational modifications, hence $\vol((L+t_0f^*H)_c):=\frac{1}{b^{2g}}\vol((\mu_b^*L+abf_b^*H)_c)$ is well-defined and
\begin{eqnarray} F'(t_0)&=&\nonumber (\dim X)\vol_{X\mid D}(\cL+t_0f^*H)\\\nonumber&\geq &  (\dim X)\vol_{X\mid D}((\cL+t_0f^*H)_c))\\&=&\nonumber (\dim X)\vol(D, (\cL+t_0f^*H)_c\mid_{D}),
\end{eqnarray}
where $D\in |f^*H|$ general and the last equality holds because $(\cL+t_0f^*H)_c$ is nef. Moreover,
for any $M\in\mathbb Z$ divisible by $b$ and $D_M\in |f_M^*H^M|$ general, we have
\begin{eqnarray}
\vol(D, (\cL+t_0f^*H)_c\mid_{D})&=&\nonumber ((\cL+t_0f^*H)_c)^{\dim X-1}\cdot D)_X\\
&=&\nonumber \frac{1}{M^{2g}}((\mu_M^*\cL+M^2t_0f_M^*H)_c^{\dim X-1}\cdot \mu_M^*D)_{X_M}\\&=& \nonumber\frac{1}{M^{2g-1}}((\mu_M^*\cL+M^2t_0f_M^*H)_c^{\dim X-1}\cdot D_M)_{X_M}\\\nonumber
&=&\frac{1}{M^{2g-1}}\vol(D_M, (\mu_M^*\cL+M^2t_0f_M^*H)_c\mid_{D_M}).
\end{eqnarray}
  Put the above inequalities together, we have
  \begin{eqnarray}\label{compare1}
  F'(t_0)\geq \frac{\dim X}{M^{2g-1}}\vol(D_M, (\mu_M^*\cL+M^2t_0f_M^*H)_c\mid_{D_M}).
  \end{eqnarray}
On the other hand,
\begin{eqnarray}\label{compare2}
D^-G(t_0)&=&\mathrm{lim}_{M\rightarrow\infty} \frac{1}{M^{2g-1}}h^0(X_M|D_M, \mu_{M}^*(\cL\otimes f^*H^{t_0})\otimes f_M^*Q)\\&=&\nonumber \mathrm{lim}_{M\rightarrow\infty} \frac{1}{M^{2g-1}}h^0(X_M|D_M, \big(\mu_{M}^*(\cL\otimes f^*H^{t_0})_c\otimes f_M^*Q\big)\mid_{D_M})\\\nonumber &\leq &\mathrm{lim}_{M\rightarrow\infty} \frac{1}{M^{2g-1}}h^0(D_M,  \big(\mu_{M}^*(\cL\otimes f^*H^{t_0})_c\otimes f_M^*Q\big)\mid_{D_M}).
\end{eqnarray}
for $Q\in\Pic^0(A)$ very general.

By (\ref{compare1}) and (\ref{compare2}), we can argue by induction on the dimension to compare the volume function $\vol(X, \cL+tf^*H)$ and the cohomological rank function $h^0_{f_*\cL}(A, t\underline{h})$.

\subsection{Clifford type inequalities}
Clifford's lemma on curve is probably the first result to compare the degrees and the dimensions of global sections of mobile divisors on smooth projective varieties. The following lemma is an application of Clifford's lemma, which is probably  known to experts.
\begin{lemm}\label{curve}
Assume that $C$ is a smooth projective curve and let $D$ be a divisor on $C$ such that $h^0(D)\geq 2$.
\begin{itemize}
\item[(1)] If $g(C)=0$, then $\frac{\deg D}{h^0(D)}\geq \frac{1}{2}$ and if $g(C)=1$, $\frac{\deg D}{h^0(D)}=1$;
\item[(2)] if $g(C)\geq 2$, then $\frac{\deg D}{h^0(D)}\geq 1$ and equality holds only when $C$ is hyperelliptic;
\item[(3)] if $C$ is neither hyperelliptic nor trigonal  and $\deg D\leq 2g(C)-2$, then
$\frac{\deg D}{h^0(D)}\geq 2-\frac{2}{g(C)}$.
\end{itemize}
\end{lemm}
\begin{proof}
Note that $(1)$ and $(2)$ follow directly from Riemann-Roch or Clifford's lemma. We just need to deal with $(3)$.

Since $h^0(D)\geq 2$, we may and will assume that the linear system $|D|$ is base point free. Let $\varphi_D$  be the morphism induced by $|D|$. Since $C$ is neither hyperelliptic nor trigonal, $g(C)\geq 5$.

If $\deg D\leq g(C)-1$,    Beauville showed in \cite[Lemma 5.1]{B} that  if $\varphi_D$ is birational, then $\frac{1}{3}(\deg D+4)\geq h^0(D)$; if $\deg \varphi_D=2$,  there exists a double cover $\pi: C\rightarrow C'$ such that $D=\pi^*D'+V$ and $\frac{1}{2}\deg D+1-g(C')\geq h^0(D)$; if $\deg \varphi_D\geq 3$, then $\deg D\geq \deg \varphi_D(h^0(D)-1)$.
In the first case, as $\varphi_D$ is birational, we have $h^0(D)\geq 3$. Hence $\frac{\deg D}{h^0(D)}\geq 2-\frac{2}{g(C)}$, unless possibly $h^0(D)=3$ and $g\geq 7$. If $\frac{\deg D}{h^0(D)}< 2-\frac{2}{g(C)}$, we  have $\deg D<6-\frac{6}{g}$. Hence $\deg D\leq 5$. But  if $\deg D\leq 5$, then the arithmetic genus of $\varphi_D(C)$ in $\mathbb P^2$ is $\leq 6$ and hence the geometric genus of $C$ is $\leq 6$ which is a contradiction.
In the second case, we have $\deg D\geq 2h^0(D)$, since $C$ is not hyperelliptic.
In the last case, we always have $\deg D\geq (2-\frac{2}{g})h^0(D)$, unless $\deg\varphi_D=3$ and $h^0(D)=2$ and $g\geq 5$, which implies that $C$ is trigonal.

We then assume that $\deg D\geq g(C)$. We may assume that $h^0(D)\geq \frac{1}{2}\deg D\geq \frac{1}{2}g$, otherwise nothing needs to be proved. If $\deg \varphi_D\geq 3$, then we still have $\deg D\geq 3(h^0(D)-1)$ and hence $\frac{\deg D}{h^0(D)}\geq 3-\frac{3}{h^0(D)}\geq2-\frac{2}{g(C)}$, since $g(C)\geq 5$. If $\deg \varphi_D=2$, then $\deg D\geq 2\deg \varphi_D(C)$. Since $C$ is not hyperelliptic, $\varphi_D(C)$ is not rational. Hence $\deg \varphi_D(C)\geq h^0(D)$ and $\deg D\geq 2h^0(D)$. We then assume that $\varphi_D$ is birational. If $h^1(D)=h^0(K_C-D)=0$ or $1$, then $$\frac{\deg D}{h^0(D)}=\frac{\deg D}{1+\deg D-g(C)+h^1(D)}\geq \frac{\deg D}{2+\deg D-g(C)}\geq 2-\frac{2}{g(C)}.$$ If $h^1(D)\geq 2$,
consider $|K_C-D|=|D'|+V$, where $V$ is the fixed divisor of $|K_C-D|$. We then apply refined Clifford's lemma \cite[P137 B1]{ACGH}, $$g-1=|K_C|\geq |D|+|D'|\geq (h^0(D)-1)+(h^1(D)-1)+1.$$ Combining this inequality with Riemann-Roch, we have $\deg D\geq 2h^0(D)-1$ and hence $\frac{\deg D}{h^0(D)}\geq 2-\frac{2}{g(C)}$.
\end{proof}

In higher dimensions, we have less precise versions of Clifford type results.
\begin{lemm}\label{general-clifford}
Let $F$ be a smooth projective variety of dimension $\geq 2$. Let $D$ be a base point free divisor on $F$ whose complete linear system $|D|$ induces a generically finite morphism $\varphi_D$ of $F$. Then
\begin{itemize}
\item[(1)] $\frac{\vol(D)}{h^0(D)}\geq \frac{1}{\dim F+1}$;
\item[(2)] If $F$ is not uniruled and $\varphi_D$ is birational, then $$\vol(D)\geq (\dim F)(h^0(D)-\dim F-1)+1;$$
\item[(3)] If $F$ is not uniruled, then $\frac{\vol(D)}{h^0(D)}\geq \frac{2}{\dim F+1}$ and the equality holds if and only if $\varphi_D: F\rightarrow \mathbb P(H^0(D))$ is surjective and is of degree $2$.
\end{itemize}
\end{lemm}
\begin{proof}
Let $\dim F=n$.  Note that $\varphi_{D}: F\rightarrow \mathbb P(H^0(D))$ is generically finite onto its image and $\varphi_D(F)$ is non-degenerate. Hence $\vol(D)=(D^n)\geq \deg_{\mathbb P(H^0(D))}(\varphi_D(F))\geq h^0(D)-n$. Thus $\frac{\vol(D)}{h^0(D)}\geq 1-\frac{n}{h^0(D)}\geq \frac{1}{n+1}$.

If $F$ is not ruled  and $\varphi_D$ is birational, we choose $H_i$, $1\leq i\leq \dim F-1$ general hyperplane section of $|D|$ and denote by $C$ its complete intersection. Then the global section of $D$ induces a birational morphism from C to a projective space $\mathbb P^{h^0(D)-n}$. We denote by $d=\vol(D)=(D\cdot C)$, $r:=h^0(D)-\dim F\geq 2$, and let $m=\lfloor \frac{d-1}{r-1}\rfloor$. By Castelnuovo's bound, $g(C)\leq \frac{m(m-1)}{2}(r-1)+m\epsilon$, where $d-1=m(r-1)+\epsilon$. On the other hand, $2g(C)-2=(K_F+(\dim F-1)D)\cdot D^{\dim F-1}\geq (\dim F-1)d$. Here we applied the main result of \cite{BDPP} which says that $F$ is not uniruled if and only if $K_F$ is pseudo-effective and hence $K_F\cdot D^{\dim F-1}\geq 0$.
Hence $m(m-1)(r-1)+2m\epsilon-2=(m-1)(d-1-\epsilon)+2m\epsilon-2\geq (\dim F-1)d$. We conclude that $m\geq \dim F$.

Assume that $F$ is not ruled and $\varphi_D$ is not birational, then $\deg \varphi_D\geq 2$  and hence $\vol(D)\geq 2(h^0(D)-n)$. If $\varphi_D$ is birational, then by $(2)$, we have $\vol(D)\geq n(h^0(D)-n-1)+1$. Combine these inequalities, we conclude that $\frac{\vol(D)}{h^0(D)}\geq \frac{2}{\dim F+1}$ and the equality holds if and only if  $\varphi_D: F\rightarrow \mathbb P(H^0(D))$ is surjective and is of degree $2$.
\end{proof}


 The following definitions appear naturally in our context.
 \begin{defi}Let $L$ be a line bundle on a smooth projective variety $F$. Assume that the global sections of $L$ define a generically finite map of $F$. For all smooth birational model $\sigma: F'\rightarrow F$, we consider all divisors $D\preceq\sigma^*L$  whose global sections define a generically finite map  of $F$ and let $\delta(L)$ to be the minimum of $\frac{\vol(F', D)}{h^0(D)}$ of all such $D$ and let $\delta_1(L)$ be the minimum of $\vol(V, D|_V)$ of all such $D$, where  $V$ is a general positive dimensional subvariety of $F'$. Here a general subvariety means a subvariety which passes through a general point of $F'$ and hence its deformation dominates $F'$.

 Moreover, we define $\delta(F)$ (resp. $\delta_1(F)$) to be the minimum of $\delta(L)$ (resp. $\delta_1(L)$) for all line bundles $L$, whose global sections define a generically finite map of $F$.
 \end{defi}

By Lemma \ref{general-clifford}, if $F$ is not uniruled, then $\delta(L)\geq \delta(F)\geq \frac{2}{\dim F+1}$ and $\delta_1(L)\geq \delta_1(F)\geq 2$.

 For convention, we will let $\delta(F)=\delta_1(F)=1$ when $\dim F=0$.

These birational invariants are related to the irrationality degree and the covering gonality (see \cite{BDELU}). We recall that the irrationality degree $\mathrm{irr}(F)$ of $F$ is defined to be
 $$\mathrm{min}\{\deg \varphi\mid \varphi: F\dashrightarrow \mathbb P^{\dim F}\mathrm{ \;is\;a\;\ dominate\;map}\}$$
and the covering gonality $\mathrm{cov.gon}(F)$ is defined to be $$\mathrm{min}\{\mathrm{gon}(C)\mid \mathrm{ an\;irreducible\; curve}\; C \subset X \;\mathrm{through\; a\; general\; point}\;x\in X
\}.$$ It is clear that $$\delta(F)\leq \frac{\mathrm{irr}(F)}{\dim F+1}$$ and $$\delta_1(F)\geq \mathrm{cov.gon}(F).$$ Indeed, we have

\begin{lemm}\label{no-interm}
 $\delta(F)\geq \min\{\frac{\mathrm{irr}(F)}{\dim F+1}, 1\}$.
\end{lemm}
\begin{proof}
After birational modifications, we may assume that a base point free divisor  $D$ on $F$ computes $\delta(F)$, namely $\delta(F)=\frac{\vol(D)}{h^0(D)}$. If $h^0(D)=\dim F+1$, then $\delta(F)=\frac{\mathrm{irr}(F)}{\dim F+1}$. If not, let $\varphi_D: F\rightarrow W\subset \mathbb P^{\dim F+s}$ be the morphism induced by the global sections of $D$, where $s\geq 1$. Let $\deg W=d_W$. Then choose $s$ general points on $W$ and take the projection from these points. We get a map of degree $d_W-s$  from $W$ to $\mathbb P^{\dim F}$. Hence $\deg\varphi_D(d_W-s)\geq \mathrm{irr}(F)$. We note that $$\delta(F)=\frac{\deg\varphi_D\cdot d_W}{\dim F+s+1}\geq \frac{d_W\mathrm{irr(F)}}{(d_W-s)(\dim F+s+1)}.$$ Thus if $d_W\leq \dim F+s+1$, we have $\delta(F)\geq \frac{\mathrm{irr}(F)}{\dim F+1}$ and if $d_W>\dim F+s+1$, then $\delta(F)>1$.
\end{proof}
 \begin{rema}\label{degrees}

 In \cite{BDELU}, the authors studied the irrationality degrees and the covering degrees of very general hypersurfaces in projective spaces. They showed that if $F$ is a smooth hypersurface of degree $d\geq n+2$ in $\mathbb P^{n+1}$, then $\mathrm{irr}(F)\geq \mathrm{cov.gon}(F)\geq d-n$ and if $F$ is a very general hypersurface in $\mathbb P^{n+1}$ of degree $d\geq 2n+1$, then $\mathrm{irr}(F)=d-1$.
 \end{rema}
  \section{Higher dimensional Severi type inequalities}
 \subsection{General line bundles}

In this section, we prove Severi type inequalities for general line bundles on an irregular variety $X$.

We consider $f: X\rightarrow A$ a primitive morphism from a smooth projective variety to an abelian variety. Let $f: X\xrightarrow{g} V\xrightarrow{h} A$ be the Stein factorization of $f$. We assume that a general fiber of $g$ is $F$, $\dim F=n$, and $\dim V=k\geq 1$.
 \begin{theo}\label{general}Under the above assumptions, let $L$ be a line bundle on $X$ such that $L_c$ is big. Then $$\vol(L)\geq \mathrm{min}\{\frac{(n+k)!}{n!}\delta(F), k!\delta_1(F)\}h^0(A, f_*L\otimes Q),$$ where $Q\in\Pic^0(A)$ general. In particular, $\vol(L)\geq k!h^0(A, f_*L\otimes Q)$ and $\vol(L)\geq 2k!h^0(A, f_*L\otimes Q)$ if $F$ is not uniruled.
 \end{theo}
 \begin{proof}
Let $F(t)=\vol(L+tf^*H)$ and $G(t)=h^0(A, f_*L\otimes H^t\otimes Q)$. Let $$t_0:=\{t\in \mathbb Q\mid (L\otimes f^*H^t)_c\;\; is\; \;big\}.$$
By assumption, we know that $t_0<0$.

 Note that when $n=0$, the above inequality comes from the usual Severi's inequality due to Barja \cite{bar} and Zhang \cite{zh1}. If $k=0$, we can conclude simply by the definition of $\delta(F)$. We then argue by double inductions on $k$ and $n$. Assume that Theorem \ref{general} has been proved when $\dim F<n$ or $\dim V<k$. Let $$N(F, m):=\mathrm{min}\{\frac{(n+m)!}{n!}\delta(F), m!\delta_1(F)\}$$ for all $m\in \mathbb Z_{>0}$. Then by Lemma \ref{no-interm}, $N(F, m)\geq m!$ and $N(F, m)\geq 2m!$ if $F$ is not uniruled.

For $t>t_0$ rational, we have by (\ref{compare1}) and (\ref{compare2}) that
 $$F'(t)\geq \frac{n+k}{M^{2g-1}}\vol(D_M, (\mu_M^*L+M^2tf_M^*H)_c\mid_{D_M})$$ and $$D^-G(t)\leq \mathrm{lim}_{M\rightarrow\infty} \frac{1}{M^{2g-1}}h^0(D_M, \big(\mu_{M}^*(L\otimes f^*H^{t})_c\otimes f_M^*Q\big)\mid_{D_M}).$$ Since $(L\otimes f^*H^{t})_c$ remains big on $X$, so is $(\mu_M^*L+M^2tf_M^*H)_c$ on $D_M$.
 Hence by induction, $F'(t)\geq (n+k)N(F, k-1)D^-G(t)$ for all $t>t_0$.
Thus we take integration from $t_0$ to $0$ and have that
\begin{eqnarray}\label{1ineq}&&\vol(L)-\vol(L+t_0f^*H)\\\nonumber &&\geq (n+k)N(F, k-1) (h^0_{f_*L}(0\underline{h})-h^0_{f_*L}(t_0\underline{h})).\end{eqnarray} After a small perturbation of $t_0$, we may assume that $t_0\in\mathbb Q$. If $h^0_{f_*L}(t_0\underline{h})=0$, nothing needs to be proved. Otherwise, $n+k>\kappa((L+t_0f^*H)_c)=k_1\geq k$. Take $M$ sufficiently large and divisible and consider the commutative diagram:
\begin{eqnarray*}
\xymatrix{
& X_M\ar[dl]_{\varphi_M}\ar[dd]^{f_M}\ar[r]^{\mu_M} & X\ar[dd]^f\\
Z_{M} \ar[dr]^{\nu_M}\\
&A\ar[r]^{\pi_M} & A,}
\end{eqnarray*}
where $\varphi_M$ is the eventual map of $\mu_M^*(L+t_0f^*H)$.
 By Lemma \ref{eventual}$, \mu_M^*(L+t_0f^*H)_c=\varphi_M^*L_1$, where $L_1$ is a line bundle on $Z_M$, $L_1=(L_1)_c$, and $$h^0(X_M, \mu_M^*(L+t_0f^*H)_c\otimes f_M^*Q)=h^0(Z_M, L_1\otimes \nu_M^*Q).$$ After birational modifications, we may also assume that $Z_M$ is smooth. By the induction  of $\dim F$, we have $$\vol(Z_M, L_1)=(L_1^m)_{Z_M}\geq k! h^0(\nu_{M*}L_1\otimes Q)= k! h^0(f_{M*}(\mu_M^*(L+t_0f^*H))\otimes Q)=M^{2g} k! h^0_{f_*L}(t_0\underline{h}).$$
Note that by the definition of $t_0$, $\kappa(L+(t_0+\epsilon)f^*H)_c)=n+k$ for any $\epsilon\in \mathbb Q_{>0}$ sufficiently small. After replacing $M$ by its multiple, we may assume that $\mu_M^*((L+(t_0+\epsilon)f^*H)_c)$ is a line bundle on $X_M$ and its eventual map is then generically finite. Thus
\begin{eqnarray*}
\vol(X, L+(t_0+\epsilon)f^*H)&=&\frac{1}{M^{2g}}\vol(X_M, \mu_M^*(L+(t_0+\epsilon)f^*H))\\&\geq &\frac{1}{M^{2g}}(\mu_M^*(L+(t_0+\epsilon)f^*H)_c^{n+k})_{X_M}\\&>&\frac{1}{M^{2g}}\big((\varphi_M^*L_1)^{k_1}\cdot \mu_M^*(L+(t_0+\epsilon)f^*H)_c^{n+k-k_1}\big)_{X_M}.
\end{eqnarray*}
Note that $F$ is still a connected component of a general fiber of $f_M$.  Hence a connected component $V$ of some general fiber of $\varphi_M$ is a general subvariety of $F$.  Moreover, the linear system of $\mu_M^*(L+(t_0+\epsilon)f^*H)_c$ defines a generically finite map of $F$. In particular, $(V\cdot\mu_M^*(L+(t_0+\epsilon)f^*H)_c^{n+k-k_1} )\geq \delta_1(F)$.
  Hence $$\big((\varphi_M^*L_1)^{k_1}\cdot \mu_M^*(L+(t_0+\epsilon)f^*H)_c^{n+k-k_1}\big)_{X_M}\geq \delta_1(F)(L_1^{k_1})_{Z_M},$$ and thus $$\vol(X, L+(t_0+\epsilon)f^*H)>\delta_1(F) k! h^0_{f_*L}(t_0\underline{h}).$$ By the continuity of volume functions, we conclude that
\begin{eqnarray}\label{2ineq}\vol(X, L+t_0f^*H)\geq\delta_1(F)k! h^0_{f_*L}(t_0\underline{h}).\end{eqnarray}
Combine (\ref{1ineq}) and (\ref{2ineq}), we have
\begin{eqnarray*}\vol(L)&\geq& \mathrm{min}\{(n+k)N(F, k-1), k!\delta_1(F)\}h^0(A, f_*L\otimes Q)\\&\geq & N(F, k)h^0(A, f_*L\otimes Q).\end{eqnarray*}
   \end{proof}

   \begin{rema}\label{variant}Let $L_F:=L|_F$.
   It is clear that  under the assumption of Theorem \ref{general}, we have $ \vol(L)\geq \mathrm{min}\{\frac{(n+k)!}{n!}\delta(L_F), k!\delta_1(L_F)\}h^0(A, f_*L\otimes Q)$.
   \end{rema}

   \begin{rema}\label{variant2}
  If $f_*L$ satisfies certain stable condition, the proof of Theorem \ref{general} gives a stronger inequality for $L$. For instance, it is easy to see that the following statement can be proved by the same argument:

  Assume that $f: X\rightarrow E$ is a surjective morphism to an elliptic curve, $L_c$ is big, and $f_*L$ is a semistable vector bundle on $E$, then $$\vol(L)\geq (\dim X)\frac{\vol(L_F)}{h^0(F, L_F)}h^0(f_*L\otimes Q).$$

   \end{rema}

   When $n=1$, the general fiber $F$ is a curve.
   \begin{coro}\label{line-bundle}
Let $f: X\rightarrow A$ be a morphism from a smooth projective variety to an abelian variety and $\dim f(X)=\dim X-1$. Denote by $d_C$ the gonality of a connected component $C$ of a general fiber of $f$. Let $\cL$ be a big line bundle on $X$ such that $\cL_c$ is also big. Then
  \begin{itemize}
  \item[(1)] $\vol(\cL)\geq (\dim X-1)!h^0(A, f_*\cL\otimes Q)$;
  \item[(2)] if $g(C)\geq 1$, $\vol(\cL)\geq \min\{(\dim X)!, d_C(\dim X-1)!\}h^0(A, f_*\cL\otimes Q);$
  \item[(3)] if $g(C)\geq 2$ and $C$ is neither hyperelliptic nor trigonal, and $\cL\mid_C\leq K_C$, then $\vol(\cL)\geq \min\{(2-\frac{2}{g(C)})(\dim X)!, d_C(\dim X-1)!\}h^0(A, f_*\cL\otimes Q).$
  \end{itemize}
\end{coro}

\begin{proof} We conclude by combining Theorem \ref{general} and Lemma \ref{curve}.
 \end{proof}

 Similarly, if  $f: X\rightarrow A$ be a morphism from a smooth projective variety to an abelian variety and a general fiber $F$ is a very general smooth hypersurface of degree $d\geq 2n+1$ in $\mathbb P^{n+1}$.  Then, we have
 \begin{coro}\label{hypersurfaces}Let $\cL$ be a line bundle on $X$ such that $\cL_c$ is big, we have $$\vol(\cL)\geq \min\{\frac{(\dim X)!}{n!}, (d-n)(\dim X-n)!\}h^0(A, f_*\cL\otimes Q).$$
 \end{coro}
 \begin{proof}
 By Remark \ref{degrees}, we know that $\mathrm{irr}(F)=d-1$ and $\mathrm{cov.gon}(F)\geq g-n$. Hence by Lemma \ref{no-interm}, $\delta(F)\geq 1$ and we also have $\delta_1(F)\geq \mathrm{cov.gon}(F)\geq g-n$. We then conclude by Theorem \ref{general}.
 \end{proof}

In practice, we also need to deal with big line bundles $L$ whose continuously globally generated part $L_c$ is not big. In this case, we denote by
 \begin{eqnarray*}
 \xymatrix{
 X\ar[r]^{\varphi_L}\ar[d]^f & Z_L\ar[dl]^{\rho}\\
 A}
 \end{eqnarray*}
 the eventual map of $L$ and let $G$ be a general fiber of $\varphi_L$.
 \begin{prop}\label{nonbig}Then $\vol(L)\geq \vol(X| G, L)k!h^0(A, f_*L\otimes Q)$, for $Q\in\Pic^0(A)$ general.
 \end{prop}
 \begin{proof}
By the construction of the eventual map, we may assume that $Z_L$ is smooth and there exists a nef line bundle $M$ on $Z_L$ such that $\varphi_L^*M\preceq L$ and $(f_*L)_c=(\rho_*M)_c$. Then $\vol(Z_L, M)\geq k!h^0(A, \rho_*M\otimes Q)=k!h^0(A, f_*L\otimes Q)$. On the other hand, we know by \cite[Proposition 2.11 and Theorem 2.13]{ELMNP} that we can compute the volume or restricted volume of a line bundle by the asymptotic intersection number and the Fujita approximation holds for the restricted volumes. To be more precise, for any $\epsilon>0$, after birational modifications, there exists a decomposition $L=A+E$, where $A$ is an ample $\mathbb Q$-divisor and $E$ is an effective $\mathbb Q$-divisor such that $(A^{\dim G}\cdot G)>\vol(X|G, L)-\epsilon$.
 We then conclude that $$\vol(L)\geq (\varphi_L^*M)^{\dim Z_L}\cdot A^{\dim G}>(\vol(X|G, L)-\epsilon)\vol(Z_L, M).$$
 \end{proof}

\begin{rema}We could compare the above proposition with a different result. Assume that   $h: X\rightarrow Y$ is a fibration between smooth projective varieties of general type with a general fiber $F$. Then Kawamata \cite[Theorem 7.1]{Zh-Ka} showed that
\begin{eqnarray*}\vol(X)\geq \binom{\dim X}{\dim Y}\vol(F)\vol(Y).
\end{eqnarray*}
The main ingredient for such a nice bound is Viehweg's weak positivity of $f_*\omega_{X/Y}^m$.
\end{rema}
 \begin{exam}
 We consider the following example of Fletcher: $Y\subset \mathbb{P}(1,3, 4, 5, 7)$ a general degree $21$ hypersurface. Then $Y$ has only terminal isolated singularities and $\omega_Y=\cO_Y(1)$. Let $Y'$ be a smooth model of $Y$, then $p_g(Y')=1$ and $\vol(Y')=\frac{1}{20}$. Let $X=Y'\times C$, where $C$ is a genus $2$ curve. Then $h^0(A_X, a_{X*}\omega_X\otimes Q)=1$ and $\vol(X)=\frac{2}{5}$. This example seems to contradict Corollary B of \cite{bar}.
 \end{exam}
 \subsection{Canonical bundle and pluri-canonical bundle}
When the line bundle $L$ is the canonical or pluricanonical bundles, we can often drop the condition on the bigness condition of  $L_c$, due to the positivity of $f_*L$.
 \begin{theo}\label{canonical-bundle-general}
Let $f: X\rightarrow A$ be a morphism from a smooth projective $m$-dimensional variety to an abelian variety with $F$ a connected component of a general fiber of $f$. Assume that $\dim F=n$.
\begin{itemize}
\item[(1)]If $|K_F|$ induces a generically finite map of $F$. Then
$$\vol(K_X)\geq \min\{\frac{m!}{n!}\delta(K_F), \delta_1(K_F)(m-n)!\}h^0(A, f_*\omega_X\otimes Q).$$
\item[(2)]If $|K_F|$ induces a map, whose general fiber $V$ is of dimension $n_1<n$, then $$\vol(K_X)\geq  \vol(F|V, K_F)(m-n)!h^0(A, f_*\omega_X\otimes Q).$$
\item[(3)] If $|rK_F|$ induces a generically finite map of $F$, for some integer $r>1$, then $$r^m\vol(K_X)\geq \min\{\frac{m!}{n!}\delta(rK_F), \delta_1(rK_F)(m-n)!\}h^0(A, f_*\omega_X^r\otimes Q).$$
\end{itemize}
\end{theo}

\begin{proof}
 The proof is almost a direct application of Theorem \ref{general}. For $(1)$, we need an extra ingredient. By the decomposition theorem in Remark (\ref{decomposition}), we know that $f_*\omega_X\otimes H^{\epsilon}$ is M-regular and hence is continuously globally generated for all $\epsilon\in \mathbb Q_{>0}$ sufficiently small. By the assumption that $|K_F|$ induces a generically finite map of $X$, we conclude that $(K_X+\epsilon f^*H)_c$ is big and thus we can apply Theorem \ref{general} to conclude that $\vol(K_X+\epsilon f^*H)\geq \min\{\frac{m!}{n!}\delta(K_F), \delta_1(K_F)(m-n)!\}h^0(A, f_*\omega_X\otimes H^{\epsilon}\otimes Q),$ for $Q\in\Pic^0(A)$ general. Let $\epsilon\rightarrow 0$, we finish the proof of $(1)$.

The proof of $(2)$ is similar. By the decomposition theorem, $f_*\omega_X\otimes H^{\epsilon}$ is M-regular and hence the eventual map of $(K_X+\epsilon f^*H)$ induces the canonical map of $F$.  Hence we apply  Proposition \ref{nonbig} to conclude.

 For $(3)$, we simply note that $f_*\omega_X^r$ is IT$^0$ (see for instance \cite{LPS}) and hence is M-regular and we conclude directly by Theorem \ref{general}.

\end{proof}

Combining Corollary \ref{line-bundle} and Theorem \ref{canonical-bundle-general}, we have
\begin{coro}\label{curve-canonical}Assume that $F$ is a curve of genus $g\geq 2$, then $$\vol(K_X)\geq \min\{m!, d_F(m-1)!\}h^0(A, f_*\omega_X\otimes Q).$$ If $F$ is neither hyperelliptic nor trigonal, then  $$\vol(K_X)\geq \min\{(2-\frac{2}{g(F)})m!, d_F(m-1)!\}h^0(A, f_*\omega_X\otimes Q).$$
If $F$ is a generic curve of genus $g\geq 5$, then $$\vol(K_X)\geq \min\{(2-\frac{2}{g(F)})m!, \lfloor\frac{g+3}{2}\rfloor(m-1)!\}h^0(A, f_*\omega_X\otimes Q).$$

\end{coro}
\begin{proof}
Only the last part needs to be explained. Indeed, if $F$ is a generic curve of genus $g\geq 5$, then $d_F=\lfloor\frac{g+3}{2}\rfloor\geq 4$. Hence we conclude by Corollary \ref{line-bundle} and Theorem \ref{canonical-bundle-general}.
\end{proof}
\begin{coro}\label{surface-canonical}
Assume that $F$ is a surface of general type. Then $$\vol(K_X)\geq (m-2)!h^0(A, f_*\omega_X\otimes Q).$$
Furthermore, unless that  $\vol(F)=1$, we will always have
$$\vol(K_X)\geq 2(m-2)!h^0(A, f_*\omega_X\otimes Q).$$

\end{coro}
\begin{proof}

We apply Theorem \ref{canonical-bundle-general}. Note that if $p_g(F)=0$, then nothing needs to be proved.

 If $|K_F|$ is generically finite, then by (1) of Theorem \ref{canonical-bundle-general},
$\vol(K_X)\geq 2(m-2)! h^0(A, f_*\omega_X\otimes Q)$.

If $|K_F|$ induces a map to a curve, and let $V$ be a general pencil of $K_F$. By (2) of Theorem \ref{canonical-bundle-general}, $\vol(K_X)\geq  \vol(F|V, K_F)\vol(Z_L, M)\geq \vol(F|V, K_F)(m-2)! h^0(A, f_*\omega_X\otimes Q)$. Let $\sigma: F\rightarrow F_0$ be the morphism from $F$ to its minimal model. Then $\vol(F|V, K_F)\geq (\sigma^*K_{F_0}\cdot V)$. By  Noether's inequality (\cite[Chapter VII, Theorem 3.1]{BHP}), we know that $K_{F_0}^2\geq 2p_g(F)-2$. Hence we conclude by Hodge index that
 $ \vol(F|V, K_F)\geq 2$, unless $\vol(K_F)=1$ and $p_g(F)=2$.
If $p_g(F)=1$, then $\vol(K_X)\geq  \vol(K_F) (m-2)! h^0(A, f_*\omega_X\otimes Q)$.

 \end{proof}

 \begin{coro}
 Assume that $F$ is a very general hypersurface of degree $d\geq 2n+1$ in $\mathbb P^{n+1}$, then $$\vol(X)\geq \min\{\frac{(\dim X)!}{n!}, (d-n)(\dim X-n)!\}h^0(A, f_*\omega_X\otimes Q).$$
 \end{coro}
 \begin{proof} This is a corollary of  Theorem \ref{canonical-bundle-general} and Corollary \ref{hypersurfaces}.
 \end{proof}
\section{Irregular varieties with small volumes}
Severi type inequalities provide a way to estimate the volumes of irregular varieties.
Note that varieties of general type in dimension $\geq 3$ could have rather small volumes due to the singularities of their minimal models. A general hypersurface $X_{46}\subset \mathbb P (4, 5, 6, 7, 23)$ is a minimal threefold with volume $\frac{1}{420}$. On the other hand,
Chen and Chen proved in \cite{CC} that if $X$ is an irregular $3$-fold, then $\vol(X)\geq \frac{1}{22}$. We will study in this section irregular varieties with small volumes via the  Severi type inequalities proved in last section.

 \subsection{Varieties of maximal Albanese dimension}
 Let $A$ be an abelian variety. We will denote by $\Phi_{\cP}: \mathrm{D}^b(A)\rightarrow \mathrm{D}^b(\Pic^0{A})$ the Fourier-Mukai transform induced by the normalized Poincar\'e line bundle $\cP$ on $A\times \Pic^0(A)$.

\begin{prop}\label{equality} Let $X$ be a smooth projective variety of maximal Albanese dimension with $\chi(X, \omega_X)=1$ and $\vol(X)=2n!$. Then the canononical model $X_{\mathrm{can}}$ is a flat double cover of a princiaplly polarized abelian variety $(A, \Theta)$ branched over a divisor $D\in |2\Theta|$.
\end{prop}
\begin{proof}
Since $\vol(X)=2n!\chi(\omega_X)$, by \cite{BPS}, we know that $a_X: X\rightarrow A_X$ is surjective and is of degree $2$. Let $a_{X*}\omega_X=\cO_A\oplus \cL$, where $\cL$ is a rank one torsion-free sheaf. Let $a_X: X\xrightarrow{\rho} \overline{X}\xrightarrow{g} A_X$. We just need to show that  the double dual $H$ of $\cL$ is a theta divisor. Indeed, $\overline{X}$ is normal and is a double cover of $A_X$, hence it is Gorenstein and $K_{\overline{X}}=g^*H$. If $H=\cO_A(\Theta)$ is a theta divisor, by the assumption that $\chi(X, \omega_X)=1$, we conclude that $\chi(A, \cL)=1$ and hence $\cL=H=\cO_A(\Theta)$ is a theta divisor on $A$. Thus we have $\rho_*\omega_X=\omega_{\overline{X}}$. Hence $\overline{X}$ has canonical singularities and is the canonical model of $X$.

We now write $\cL=H\otimes\cI_Z$, where $H$ is an ample line bundle on $A$. The idea is to compare the two functions $\vol(K_X+ta_X^*H)$ and $h^0_{a_{X*}\omega_X}(t\underline{h})$. By \cite[Page 11, proof of Theorem 1.2]{BPS2}, we know that $\vol(K_X+tH)=2n!h^0_{a_{X*}\omega_X}(t\underline{h})$ for $t\leq 0$. Note that $\chi(X, \omega_X)=1$, hence $\chi(A, H\otimes \cI_Z)=1$. Since $\cL=H\otimes\cI_Z$ is also M-regular, we know that $\Phi_{\cP}(\cL^{\vee})=R^n\Phi_{\cP}(\cL^{\vee})[-n]$ is a shifted rank $1$ torsion-free sheaf by \cite[Corollary 3.2]{pp2}. Hence $R^0\Phi_{\cP}(\cL)=\mathcal{H}om(R^n\Phi_{\cP}(\cL^{\vee}),\cO_A)$ is a line bundle (see for instance \cite[Lemma 2.2]{pp3}). Let $\varphi_H: A\rightarrow \Pic^0(A)$ be the isogeny induced by $H$.  Then, from the short exact sequence $$0\rightarrow\varphi_H^*R^0\Phi_{\cP}(\cL)\rightarrow \varphi_H^*R^0\Phi_{\cP}(H)=(H^{-1})^{\oplus \chi(H)},$$ we conclude that $R^0\Phi_{\cP}(\cL)=H_1$ is a negative line bundle. Indeed, $H_1^{-1}\otimes H^{-1}$ is effective. By the main result of \cite{JP}, we know that
\begin{eqnarray*}
h^0_{a_{X*}\omega_X}(t\underline{h})=\frac{(-t)^n}{\chi(H)}\chi(\varphi_H^*R^0\Phi_{\cP}(\cL)\otimes H^{\frac{-1}{t}})=\frac{1}{n!\chi(H)}(H-tH_1)_A^n
\end{eqnarray*}
for $t<0$ sufficiently close to $0$. On the other hand, $a_X: X\rightarrow A_X$ also factors through the canonical model of $X$. We write $a_X: X\xrightarrow{\sigma} X_{\mathrm{can}}\xrightarrow{g_1} A_X$. Then $$\vol(K_X+ta_X^*H)=\vol(K_{X_{\mathrm{can}}}+tg^*H)=(K_{\mathrm{can}}+tg_1^*H)_{X_{\mathrm{can}}}^n$$ for $t<0$ sufficiently close to $0$. Comparing the coefficients of $t^n$ and $t^{n-1}$, we have  $$(-H_1)_A^n=\chi(H)(H^n)_A$$ and $$(K_{\mathrm{can}}\cdot g_1^*H^{n-1})_{X_{\mathrm{can}}}=\frac{2}{\chi(H)}(H\cdot (-H_1)^{n-1})_A.$$ Note that we have the natural birational morphism $X_{can}\rightarrow \overline{X}$, hence $(K_{\mathrm{can}}\cdot g_1^*H^{n-1})_{X_{\mathrm{can}}}=(K_{\overline{X}}\cdot f^*H^{n-1})_{\overline{X}}=2(H^n)_A$. Thus $(-H_1)^n_A=(H\cdot (-H_1)^{n-1})_A$. Hence $H_1^{-1}$ is algebraically equivalent to $H$. Note that $\varphi_H^*R^0\Phi_{\cP}(\cL)=H^{-1}$, where $\deg \varphi_H=\chi(H)^2$. Hence $(-1)^n\chi(H)=\chi(H)^2\chi(R^0\Phi_{\cP}(\cL))$ and we conclude that $\chi(H)=1$ and $H=\cO_A(\Theta) $ is a theta divisor.
\end{proof}
\begin{coro} Let $X$ be a smooth projective variety of general type with maximal Albanese dimension. Then $\vol(X)\geq 2n!$ and equality holds if and only if the canonical model of $X$ is a flat double cover of a principally polarized abelian variety $(A, \Theta)$ branched over a divisor $D\in |2\Theta|$.
\end{coro}
\begin{proof}
By the last proposition, we just need to deal with the case that $\chi(X, \omega_X)=0$.

When $\dim X=3$, by the main result of \cite{CDJ}, the canonical model $f: X_{\mathrm{can}}\rightarrow A_X$ is a flat $\mathbb Z_2^2-$cover, there exists an isogeny $\rho: A_X\rightarrow E_1\times E_2\times E_3$ and $a_*\omega_X=\cO_A\oplus \cL_1\oplus \cL_2\oplus\cL_3$, where $\cL_i$ are line bundles on $A_X$ whose Iitaka fibration is the natural morphism $A_X\rightarrow E_j\times E_k$, where $\{i, j, k\}=\{1, 2, 3\}$. Moreover, $2K_{\mathrm{can}}=f^*(\cL_1\otimes \cL_2\otimes\cL_3)$ is a line bundle. We conclude that $\vol(X)\geq 24$.

In higher dimensions, such characterization is not available. Let $\dim X=n$. We argue by induction on dimensions to prove that $\vol(X)>2(\dim X)!$ when $\chi(\omega_X)=0$.

Assume that $\vol(Y)>2(\dim Y)!$ for all $Y$ smooth projective of maximal Albanese dimension and of general type of dimension $<n$, and $\chi(\omega_Y)=0$. We apply Setting 3.2 of \cite{JLT} of $X$. Namely, there exists a codimension-$k$ component $[Q]+\PB$ of $V^k(\omega_X)$ with $0<k<n$ maximal, where $\PB$ is an abelian subvariety of $\Pic^0(A_X)$, $Q$ is a torsion line bundle and $[Q]\in \Pic^0(A_X)$ is the corresponding point. Taking the natural quotient $A_X\rightarrow \Pic^0\PB:=B$ and taking Stein factorization, we have the following commutative diagram
\begin{eqnarray*}
\xymatrix{
X\ar[r]^{a_X}\ar[d]^h & A_X\ar[d]\\
Y\ar[r] & B,}
\end{eqnarray*}
where after birational modifications, all varieties are smooth. Let $F$ be a general fiber of $h$. Note that by \cite{JLT}, we know that there are two cases, either $Q$ is the trivial line bundle or $Q$ is torsion and non-trivial.

If $Q$ is the trivial line bundle, then $Y$ is of general type. Hence by Kawamata's result \cite[Theorem 7.1]{Zh-Ka}, $\vol(X)\geq \binom{\dim X}{\dim Y}\vol(Y)\vol(F).$ Hence by induction, we have  $\vol(X)\geq 4n!$ in this case.

 In the second case, $X\rightarrow B$ is a fibration. We know that $\Pic^0(B)$ is the neutral component of $H:=\ker(\Pic^0(X)\rightarrow \Pic^0(F))$. We take a finite group $G\subset H$ such that $G+\Pic^0(B)=H$ and $G\cap \Pic^0(B)=[\cO_X]$. Let $\tilde{A}_X\rightarrow A_X$ be the \'etale cover induced by $G$ and let $\tilde{X}\rightarrow X$ be the corresponding base change and let  $\tilde{X}\rightarrow \hat{Y}\rightarrow B$ be the Stein factorization of the natural morphism $\tilde{X}\rightarrow B$. After birational modifications, we may assume that $\hat{Y}$ is smooth. We have the commutative diagram
 \begin{eqnarray*}
\xymatrix{
\tilde{X}\ar[r]^{\pi}\ar[d]^{q} & X\ar[d]^p\\
\hat{Y}\ar[r]^{\rho} & B.}
\end{eqnarray*}
Then $F$ is still a general fiber of $q$ and $\rho$ is a birational $G$-cover and $$\rho_*\omega_{\hat{Y}}=\bigoplus_{[Q]\in G}\cL_Q,$$ where $\cL_Q=R^kp_*(\omega_X\otimes Q)$ and $\cL_Q$ is M-regular for all $[Q]$ non-zero due to the maximality of $k$ (see \cite[Section 3]{JLT}). Hence $\chi(\omega_{\hat{Y}})\geq |G|-1$ and $\vol(\hat{Y})\geq 2(|G|-1)(\dim \hat{Y})!$  Hence
\begin{eqnarray*}\vol(X)&=&\frac{1}{|G|}\vol(\tilde{X})\geq \frac{1}{|G|}\binom{\dim X}{\dim \hat{Y}}\vol(\hat{Y})\vol(F)\\
&\geq & \frac{4(|G|-1)}{|G|}(\dim X)!\geq 2(\dim X)!.
\end{eqnarray*}
 Note that if $\vol(X)=2(\dim X)!$, then $|G|=2$, $\vol(\hat{Y})=2(\dim \hat{Y})!$, and $\vol(F)=2(\dim F)!$.
  In the following we will argue by contradiction to see that we cannot have all the equalities.

  By induction, the canonical model of $F$ is a flat double cover of $A_F$ branched over $D_F\in|2\Theta_F|$, where $\Theta_F$ is a principal polarization on $A_F$. Let $K$ be the kernel of $\tilde{A_X}\rightarrow B$. By  the construction of $\tilde{X}$, the natural morphism $F\rightarrow K$ is primitive. Hence $K\simeq A_F$ and $F\rightarrow K$ is of degree $2$. Thus  $\tilde{X}\rightarrow \tilde{A}_X$ is of degree $4$ and there exists a birational involution $\sigma_1$ acting on $\tilde{X}$ whose restriction on $F$ is the canonical involution.

 We claim that $\tilde{X}\rightarrow \tilde{A}_X$ is indeed a birationally $(\mathbb Z_2\times \mathbb Z_2)$-cover. Indeed, since $\chi(\omega_X)=0$, the decomposition of $a_{X*}\omega_X$ has a special form. Indeed, $(a_{X*}\omega_X)_c=0$. By the same argument as above, we conclude that each direct summand of $a_{X*}\omega_X$ is indeed a nef line bundle and hence after possibly a further abelian \'etale cover, $\tilde{X}\rightarrow \tilde{A_X}$ has $3$ independent involutions. Thus $\tilde{X}\rightarrow \tilde{A}_X$ is indeed a birationally $(\mathbb Z_2\times \mathbb Z_2)$-cover. Hence so is $X\rightarrow A_X$. But since $X\rightarrow B$ is a fibration, we conclude that $F$ is a birationally $(\mathbb Z_2\times \mathbb Z_2)$-cover of the kernel of $A_X\rightarrow B$, which is obviously a contradiction to the assumption that $\vol(F)=2(\dim F)!$.
\end{proof}
\subsection{Irregular threefolds}

For general irregular varieties, the picture is not clear. We focus on irregular threefolds and have some partial results.
\begin{prop}\label{fiber1}
Let $X$ be a $3$-fold of general type, of Albanese fiber dimension $1$. Then $\vol(X)\geq 2$.
\end{prop}
\begin{proof}
By Corollary \ref{curve-canonical}, $\vol(X)\geq 4h^0(A_X, a_{X*}\omega_X\otimes Q)$ for $Q\in\Pic^0(X)$ general. We still need to deal with the case that $h^0(A_X, a_{X}*\omega_X\otimes Q)=0$ for $Q\in\Pic^0(X)$ general.   In this case, the M-regular part of $a_{X*}\omega_X $ is $0$. By \cite{JS}, we know that the translates through the origin of all irreducible component of $V^0(a_{X*}\omega_X)$ generate $\Pic^0(X)$. Hence $a_X$ is surjective and $V^0(a_{X*}\omega_X)$ contains torsion translates of at least two different elliptic curves.

We may write $a_{X*}\omega_X=\bigoplus_{p_i: A\rightarrow E_i}\bigoplus_j p_i^*\cF_{ij}\otimes Q_{ij}$, where $\cF_{ij}$ are ample vector bundles on $E_i$ and $Q_{ij}$ are torsion line bundles. Let $E=E_1$, $Q:=Q_{11}^{-1}$, $f:=p_1\circ a_X: X\rightarrow E$. Then $f_*(\omega_X\otimes Q)$ is a vector bundle on $E$, which is the direct sum of $\cF_{11}$ and some torsion line bundles on $E$. In particular, $h^0(E, f_*(\omega_X\otimes Q)\otimes P)\neq 0$ for $P\in\Pic^0(E)$ general.

 Moreover, we see that a general fiber $X_t$ of $f$ is an irregular surface, where $t\in E$.
 We consider $a_t: X_t\rightarrow K$ the corresponding fiber of $f$. Hence $a_{t*}\omega_{X_t}$ has a direct summand $Q_{11}\mid_K$. Hence an abelian \'etale cover of $X_t$ is of maximal Albanese dimension. The Severi inequality for surfaces $\vol(K_{X_t})\geq 4\chi(\omega_{X_t})\geq 4$ holds for $X_t$. In particular, by a small variant of Corollary \ref{surface-canonical}, we have $\vol(K_X\otimes Q)\geq 2h^0(E, f_*(\omega_X\otimes Q)\otimes P)$ for $P\in\Pic^0(E)$ general. Hence $\vol(X)=\vol(K_X\otimes Q)\geq 2$.

\end{proof}

\begin{theo}\label{irregular-3-volume}
Let $X$ be an irregular threefold, then $\vol(K_X)\geq \frac{3}{8}$.
\end{theo}
\begin{proof}
By Proposition \ref{equality} and Proposition \ref{fiber1}, we just need to deal with the case that $X$ is of Albanese fiber dimension $2$. By Corollary \ref{surface-canonical}, if $h^0(a_{X*}\omega_X\otimes Q)>0$ for $Q\in\Pic^0(X)$ general, then $\vol(X)\geq 1$. Hence we will assume that $h^0(A_X, a_{X*}\omega_X\otimes Q)=0$ and in particular, $a_X: X\rightarrow A_X$ is a fibration onto an elliptic curve. Let $F$ be a general fiber of $a_X$.

We first assume that $|2K_F| $ induces a generically finite map of $F$. Then by Theorem \ref{general}, $$8\vol(K_X)\geq 2h^0(A_X, a_{X*}\omega_X^2\otimes Q).$$

 If $h^0(A_X, a_{X*}\omega_X^2\otimes Q)\geq 2$, we are done.

 If not, $a_{X*}\omega_X^2$ is a degree $1$ vector bundle on $A_X$. Hence $a_{X*}\omega_X^2$ is stable and we may apply the observation in Remark \ref{variant2} to conclude that $$\vol(K_X)\geq \frac{3\vol(K_F)}{2(\chi(\cO_F)+\vol(K_F))}.$$ We just need to show that $\frac{3\vol(K_F)}{\chi(\cO_F)+\vol(K_F)}\geq \frac{3}{4}$. If $q(F)=0$, then by Noether's inequality, $\vol(K_F)\geq 2p_g(F)-4$ and if $q(F)>0$, by Debarre's inequality \cite{D1}, $\vol(K_F)\geq 2p_g(F)$.
 Hence if $q(F)\neq 0$, $\frac{3\vol(K_F)}{\chi(\cO_F)+\vol(K_F)}\geq 2$. If $q(F)=0$, we have $\frac{3\vol(K_F)}{\chi(\cO_F)+\vol(K_F)}\geq 1$ unless $p_g(F)=2$ and $\vol(K_F)=1$.  In this case, $\vol(K_X)\geq \frac{3}{8}$.

If $|2K_F|$ is not generically finite, then by \cite[Section VII, Theorem 7.4, Theorem 7.6]{BHP}, we know that $p_g(F)=q(F)=0$ and $\vol(K_F)=1$. We also know that a general fiber of the map of $|2K_F|$ is a curve $C$ with genus $3$ or $4$ (see \cite[Theorem 5.1]{CP}) and hence the linear system $|2K_{F_0}|$ has no base divisors and $\sigma_*C\in|2K_{F_0}|$, where $\sigma: F\rightarrow F_0$ is the morphism from $F$ to its minimal model. Hence by Proposition \ref{nonbig}, $\vol(2K_X)=8\vol(K_X)\geq 2\vol(X|C, K_X)$. In this case, $\vol(X|C, K_X)=\vol(F|C, \sigma^*K_{F_0})=2$. Hence $\vol(K_X)\geq \frac{1}{2}$.
\end{proof}

\end{document}